\documentclass[a4]{amsart}
\usepackage{amssymb,amsfonts}
\usepackage{euscript,mathrsfs}
\usepackage{latexsym}
\usepackage{amscd}
\usepackage{amsmath}
\usepackage{a4}
\usepackage{epsf}
\usepackage{color}
\usepackage{graphicx}
\usepackage{amsopn}
\usepackage{mathptmx}
\usepackage{bbm}
\usepackage{microtype}

\providecommand{\abs}[1]{\left\lvert#1\right\rvert} 

\newcommand{\comment}[1]{ }

\renewcommand{\subset}{\subseteq}  
\renewcommand{\supset}{\supseteq}
\newcommand{\defas}{\mathrel{\mathop{:}}=}   
\DeclareMathOperator*{\bigtimes}{\textnormal{\Large $\times$}} 
\DeclareMathOperator*{\supp}{supp}  
\DeclareMathOperator*{\conv}{conv}  
\newcommand{\set}[1]{\left\lbrace #1 \right\rbrace} 

\theoremstyle{plain}
\newtheorem{theorem}{Theorem}
\newtheorem{lemma}[theorem]{Lemma}

\newtheorem*{corollary}{Corollary}

\theoremstyle{definition}
\newtheorem{definition}[theorem]{Definition}

\newtheorem{example}[theorem]{Example}

\theoremstyle{remark}
\newtheorem*{remark}{Remark}

\begin{document}

\title{Neighborliness of Marginal Polytopes}

\author{Thomas Kahle}

\address{Max Planck Institute for Mathematics in the Sciences\\
  Inselstrasse 22 \\
  D-04103 Leipzig, Germany \\
  kahle@mis.mpg.de
}

\date{\today}

\begin{abstract}
  A neighborliness property of marginal polytopes of hierarchical
  models, depending on the cardinality of the smallest non-face of the
  underlying simplicial complex, is shown. The case of binary variables
  is studied explicitly, then the general case is reduced to the
  binary case. A Markov basis for binary hierarchical models whose
  simplicial complexes is the complement of an interval is given.

  \keywords{Marginal Polytopes \and Markov Bases \and Hierarchical Models} 
  
\end{abstract}

\maketitle

\section{Introduction}
\label{sec:introduction}
The marginal polytope is an interesting combinatorial object that
appears in statistics\cite{hostensullivant02a}, coding theory
\cite{wainwright03:_variat_infer_in_graph_model,feldman03:_using_linear_progr_to_decod_linear_codes,KahleWenzelAy08}
and, under a different name, in toric algebra \cite{fulton93toric}. It
encodes in its face lattice the complete combinatorial information
about the boundary of certain statistical models. To define it we have
to take a very brief excursion to statistics, namely the theory of
hierarchical models for contingency tables.

Consider a collection of $n$ random variables taking values in finite
sets $\mathcal{X}_{i}, i=1,\ldots,n$. We denote $N\defas
\set{1,\ldots,n}$, and its power set as $2^{N} \defas \set{B: B
  \subset N}$. For a subset $B\subset N$ of the variables, we denote
its set of values as $\mathcal{X}_B \defas \bigtimes_{i\in
  B}\mathcal{X}_{i}$, and abbreviate $\mathcal{X} \defas
\mathcal{X}_{N}$. We have the natural projections
\begin{equation}
  \begin{aligned}
    X_B : \mathcal{X} & \to \mathcal{X}_B \\
    (x_i)_{i\in N} & \mapsto (x_i)_{i\in B} =: x_{B}.
  \end{aligned}
\end{equation}
We slightly abuse notation and denote $x_{B}$ the projection of $x$,
which is a function of $x$, and by the same symbol an arbitrary
element $x_{B} \in \mathcal{X}_{B}$. A contingency table is a
function $u :{\mathcal{X}} \to \mathbb{N}_{0}$. It is thereby a vector
in the space $ \mathbb{N}_{0}^{\mathcal{X}}$. For $B \subset N$ we
define the marginal table $u_B \in \mathbb{N}_{0}^{\mathcal{X}_B}$ as
the vector with components
\begin{equation}
  u_B(x_B) \defas \sum_
  {y : X_{B}(y) = x_{B} }  u(y).
\end{equation}
A so called hierarchical model for contingency tables can be given by a
simplicial complex $\Delta$ on the set $N$ of variable indexes
\cite{hostensullivant02a,develinsullivant03}. The facets $\mathcal{F}$
of $\Delta$ are defined as the inclusion maximal faces. They determine
the marginal map:
\begin{equation}
\begin{aligned}
  \pi_\Delta : \mathbb{R}^{\mathcal{X}} & \to \bigoplus_{F \in
    \mathcal{F}} \mathbb{R}^{\mathcal{X}_F} \\
  u & \mapsto \left(u_F\right)_{F \in \mathcal{F}}.
\end{aligned}
\end{equation}
It is a linear map computing all marginal tables corresponding to
facets.  We define cylinder sets denoting for $B\subset N$, and $y_B \in
\mathcal{X}_B$
\begin{equation}
  \set{X_B = y_B} \defas \set {x \in \mathcal{X} : X_{B}(x) = y_B}. 
\end{equation}
With respect to the canonical basis, the matrix representing
$\pi_\Delta$ is the $d \times \abs{\mathcal{X}}$ matrix
\begin{equation}
\label{eq:A-matrix-defn}
  A_{\Delta} \defas (A_{(B,y_B),x})_{(B,y_B),x} \text{ where } A_{(B,y_B),x}
  \defas   \begin{cases}
    1 & \text{ if  } X_{B}(x) = y_B \\
    0 & \text{ otherwise}.
  \end{cases} 
\end{equation}
The rows of this matrix are indexed by pairs $(B,y_B)$, where $B \in
\mathcal{F}$ is a facet of $\Delta$ and $y_B \in \mathcal{X}_B$ is a
configuration on $B$. Then $d$ is defined as the number of such
pairs. If the simplicial complex is clear, we will sometimes omit the
index $\Delta$.

\begin{definition}[Marginal Polytope]
  The \emph{marginal polytope} is the convex hull of the columns of
  $A_{\Delta}$:
  \begin{equation}
    \label{eq:def_marginal_polytope}
    Q_{\Delta} \defas \conv \set{A_{x} : x \in \mathcal{X}} 
    \subset \mathbb{R}^{d}.
  \end{equation}
\end{definition}

\begin{example}[Two independent binary variables] \sloppy
  In the case of two binary variables, we have $\mathcal{X} =
  \set{(00),(01),(10),(11)}$. Let $\Delta = \set{\set{1},\set{2}}$,
  then the matrix $A_{\Delta}$ is given as
  \begin{equation}
  A_{\Delta} = 
  \begin{pmatrix}
    1 & 1 & 0 & 0 \\
    0 & 0 & 1 & 1 \\
    1 & 0 & 1 & 0 \\
    0 & 1 & 0 & 1 
  \end{pmatrix}.
  \end{equation}
  The columns are ordered as
  $(\set{1},0),(\set{1},1),(\set{2},0),(\set{2},1)$. If $\Delta$ was
  the whole power set, $A_{\Delta}$ would be the $4 \times 4$ identity matrix. 
  The marginal polytopes are easily identified as a 2-dimensional
  square and a 3-dimensional simplex respectively.
\end{example}
Our object of interest is the toric ideal:
\begin{equation}
  \mathcal{I}_\Delta \defas \left\langle p^{u} - p^{v} :
    u,v\in \mathbb{N}_{0}^{\mathcal{X}},\pi_\Delta(u) = 
    \pi_\Delta (v)\right\rangle.
\end{equation}
Here, we used the standard notation for monomials in the variables
$p_{x}, x\in \mathcal{X}$, namely $p^u \defas \prod_{x\in \mathcal{X}}
p_x^{u(x)}$. Throughout the whole paper we use the convention that
$0^{0} = 1$. The set of indexes with non-vanishing exponent will be
called the \emph{support} of the binomial $\supp (p^{u} - p^{v})
\defas \set{x \in \mathcal{X} : u(x) + v(x) > 0}$. The supports of $u$
and $v$ will also be called the \emph{positive} respectively
\emph{negative support} of the binomial. The ideal
$\mathcal{I}_{\Delta}$ is a homogeneous prime ideal in the polynomial
ring $\mathbb{C}[p_{x} : x\in\mathcal{X}]$. In statistics the
restriction of the corresponding variety to the non-negative real
cone, would be called the closure of an \emph{exponential
  family}. This seminal observation is the cornerstone of what is now
called algebraic statistics
\cite{diaconissturmfels98,geigermeeksturmfels06,pachter05:_algeb_statis_for_comput_biolog}.

A first task is to find a suitable finite generating set of this
ideal. Very useful is a Markov basis defined as follows:
\begin{definition}
  A finite set $M \subset \ker_{\mathbb{Z}} \pi_\Delta$ is
  called a \emph{Markov basis} for the hierarchical model $\Delta$ if
  for each two contingency tables $u,v \in \mathbb{N}_{0}^{\mathcal{X}}$
  with equal marginals $\pi_\Delta(u) = \pi_\Delta(v)$ there exists a
  sequence $m_i, i=1,\ldots,l$ in $ \pm M$ such that
  \begin{equation}
    u = v +\sum_{i=1}^l m_i,
  \end{equation}
  where
  \begin{equation}
    v + \sum_{i=1}^k m_i \in \mathbb{N}_{0}^{\mathcal{X}} \text{ for all
      $k=1,\ldots,l$}.
  \end{equation}
\end{definition}
The crucial property of a Markov basis is that any two tables, having
the same marginals, can be connected without leaving the non-negative
cone. A key theorem is, that exactly a Markov basis gives
the desired set of generators:
\begin{theorem}[\cite{diaconissturmfels98}]
\label{sec:sturmfelsmarkovtheorem}
  A finite set $M$ is a Markov basis if and only if
  \begin{equation}
    \mathcal{I}_\Delta = \left\langle p^{m^+} - p^{m^-} : m \in M \right\rangle,
  \end{equation}
  where $m^+(x) \defas \max \set{0,m(x)}$, $m^-(x) \defas \max
  \set{0,-m(x)}$, such that $m = m^+ - m^-$.
\end{theorem}
The elements in a Markov basis are referred to as \emph{Markov moves}.
In the following section we will give a lower bound on the cardinality
of the positive and negative support of any move.

\section{A lower degree bound}
\label{sec:main-theorem}
\begin{theorem}
\label{sec:main-theorem-thm}
Let $\Delta$ be a simplicial complex on $N$ and $\mathcal{I}_\Delta$
the corresponding toric ideal. Let $g$ be the minimal cardinality of a
non-face of $\Delta$.  Each generator of $\mathcal{I}_{\Delta}$ has
degree at least $2^{g-1}$. Moreover, the positive and negative
supports of each generator both have cardinality bigger or equal to
$2^{g-1}$.  The degree bound is realized only by square free
binomials.
\end{theorem}

\begin{remark}
  Note that we give a lower bound on the -smallest- degree among the
  generators. Lower bound on the largest degree have been considered
  for a measure of complexity of the model for instance in
  \cite{geigermeeksturmfels06}. There, it is shown that one finds a
  simplicial complex on $2n$ units, such that there exists a generator
  of degree $2^{n}$. Furthermore, in \cite{develinsullivant03} the
  authors study an algorithm which, for graph models, computes all
  generators of a given degree. Finally, in
  \cite{loera06:_markov_bases_of_three_way} the case of 2-margins of
  $(r,s,3)$-tables is studied. It is shown that as $r$ and $s$ grow
  the support and degree of a maximal generator cannot be
  bounded. This has interesting implications for data disclosure.
\end{remark}

\begin{remark}[Graph models]
  A graph model is a hierarchical model for which $\dim \Delta \leq 1$
  holds.  If its graph is not complete, the bound reduces to the
  trivial bound $\deg m \geq 2$. On the other hand, for the complete
  graph, there are no quadratic generators.
\end{remark}

\begin{remark}[Type of generators]
  The vectors that achieve the bound (see Lemma
  \ref{sec:supp-mark-basis-eA-lemma}) are natural generalizations of
  the quadratic Markov moves for the independence model
  \cite{diaconissturmfels98}.
\end{remark}

We will prove Theorem \ref{sec:main-theorem-thm} in two steps. First,
the binary case is studied explicitly. Then the general case is
reduced to the binary case.

\subsection{The binary case}
\label{sec:binary-case}
In this section we have $\mathcal{X} = \set{0,1}^{N}$. This will allow
us to use a special orthogonal basis of $\ker_\mathbb{Z}
A_{\Delta}$. Using this, we find that any element in the kernel
has a lower bound for the cardinality of its support.

Put $\Delta^{c} \defas 2^{N} \setminus \Delta$ the set of
non-faces of $\Delta$. For elements $G \in \Delta^{c}$ we define the
upper intervals
\begin{equation}
  \left[G, N\right] \defas \set{B \subset N : B \supset G}
\end{equation}
which are contained in $\Delta^{c}$.  Next, for each $B \subset N$ we
define a vector $e_{B} \in \mathbb{R}^{\mathcal{X}}$ with components:
\begin{equation}
  e_{B} (x) \defas (-1)^{E(B,x)}
\end{equation}
where $E(B,x) \defas\abs{\set{i\in B : x_{i} = 1}}$ is the number of
entries equal to one that $x$ has in $B$. Observe, that $e_{B}$ depends on its
argument only through $x_{B}$, the part in $B$.  Therefore we will
sometimes abuse notation and write $e_{B}(x_{B})$ for the value of
$e_{B}$ at any configuration which projects to $x_{B}$. We have
\begin{lemma}[\cite{kahleay06}] The set $\set{e_{B} : B\subset N}$ is an orthogonal basis
  of $\mathbb{R}^{\mathcal{X}}$ such that $\set{e_{B} : B \in
    \Delta^{c}}$ is a basis of $\ker_{\mathbb{Z}} A_{\Delta}$. 
\end{lemma}
\begin{remark}[Characters]
  If we treat $\mathcal{X}$ as the additive group
  $\left(\mathbb{Z}\slash 2\mathbb{Z}\right)^{n}$ then the characters
  of this group form an orthonormal basis (with respect to the product
  induced by the Haar measure, which in this case is proportional to
  the standard product) of $\mathbb{C}^{\mathcal{X}}$. The characters
  are exactly given by the vectors $e_{B},B \subset N$. In our case
  the characters are real functions and also a basis of
  $\mathbb{R}^{\mathcal{X}}$.  See \cite{kahleay06,pontryagin66} for
  details.
\end{remark}

\begin{lemma}
\label{sec:supp-mark-basis-eA-lemma} Let $G \in\Delta^{c}$ and
$\mathcal{G}\defas \left[G,N\right]$. For $g \defas \abs{G}$ it holds
  \begin{equation}
    m_{\mathcal{G}}^{0}(x) \defas\sum_{B\in \mathcal{G}} e_{B}(x) =
    \begin{cases}
      2^{n - g}\, e_{G}(x_{G}) & \text{ if } x_{N
        \setminus G} = (0,\ldots,0) \\
     0 & \text{ otherwise }.
    \end{cases}
  \end{equation}
  Furthermore, for any $x_{C} \in \mathcal{X}_{C}$ we have the
  identity
  \begin{equation}
    \label{eq:eb-sum-identity}
    \sum_{x\in \set{X_C = y_C}} e_B(x) =
    \begin{cases}
      2^{n-\abs{C}} e_B(y_C) & \text{ if } B\subset C\\
      0 & \text{otherwise.}
    \end{cases}
  \end{equation}
\end{lemma}
\begin{proof}
  For the second case assume we have $i\in N\setminus G$ such that
  $x_i=1$. Since half of the sets in $[G,N]$ contain $i$, while the
  other half does not contain $i$, it follows that the sum equals zero
  if such an $i$ exists. The first case is now clear. All the summands
  are equal to $e_G$ in this case, and there are exactly $2^{n-g}$
  terms. The identity (\ref{eq:eb-sum-identity}) follows by by the
  same argument. 
\end{proof} 
\begin{remark}
  By choosing appropriate signs in the sum, one can achieve any of the
  cylinder sets $\set{X_{N\setminus G} = x_{N\setminus G}}$
  instead of $\set{X_{N\setminus G} = 0}$ as the support. To be
  concrete, we have
  \begin{equation}
  \begin{aligned}
    \label{eq:markov-moves} m_{\mathcal{G}}^{y_{N\setminus G}}(x) &
    \defas \sum_{B\in \mathcal{G}} (-1)^{E(B,y_{N \setminus
        G})}e_{B}(x) \\ & = \begin{cases}
      2^{n - g}\, e_{G}(x_{G}) & \text{ if } x_{N \setminus G} = y_{N\setminus G} \\
      0  & \text{otherwise}.
      \end{cases}
    \end{aligned}
  \end{equation}
\end{remark}

The vectors we have just constructed have minimal support.  In
the following we will deduce a technical, but elementary statement
about large subsets of $\mathcal{X}$. In Lemma
\ref{sec:positive-support-is-bounded-from-below}, it will follow that
choosing $G$ minimal in $\Delta^{c}$, the value $2^{n} - 2^{\abs{G}}$,
as in Lemma \ref{sec:supp-mark-basis-eA-lemma}, is the maximal number
of zeros, which can be achieved by non-trivial linear combinations of
the vectors $e_B, B \in \Delta^{c}$. 
\begin{lemma}\label{sec:supp-mark-basis-card-lemma}
  Let $g \in \set{1,\ldots, n}$ be fixed.  
  For $\mathcal{Y} \subset \mathcal{X}$ with $\abs{\mathcal{Y}} > 2^n
  - 2^{g}$ the following statement holds:
  \begin{itemize} 
  \item For each $B \subset N$ with $\abs{B}\geq g$,
    $\mathcal{Y}$ contains one of the cylinder sets $\set{X_B =
      x_{B}}$. More formally: $\exists x_{B} \in \mathcal{X}_B$
    such that $\set{X_B = x_{B}} \subset \mathcal{Y}$.
  \end{itemize}
\end{lemma}
\begin{proof}
  The statement follows from a simple cardinality argument.  Assume
  the contrary, let $B$ be given, and $\forall x_B \in
  \mathcal{X}_B,\, \exists x \in \mathcal{X} \setminus \mathcal{Y}$
  such that $x_{B} = X_{B}(x)$. These $x$ are all distinct, since they
  differ on $B$. We find $\abs{\mathcal{Y}} \leq 2^{n} - 2^{g}$. 
\end{proof}

\begin{lemma}
\label{sec:positive-support-is-bounded-from-below}
  Let $g$ denote the minimal cardinality among the sets in
  $\Delta^{c}$. Then any non-zero linear combination of the vectors
  $e_{B}, B\in \Delta^{c}$ has at least $2^{g-1}$ positive and
  $2^{g-1}$ negative components.
\end{lemma}
\begin{proof}
  Assume we have a linear combination
  \begin{equation}
    m = \sum_{B \in \Delta^{c}} z^{B} e_{B} \in \ker\pi_{\Delta}
  \end{equation}
  which has less then $2^{g-1}$ positive components. It has at least
  $2^{n}-2^{g-1} + 1$ non-positive components. Let $\mathcal{Y}_{\leq}
  \subset \mathcal{X}$ denote the corresponding indexes. Let $G
  \in \Delta^{c}$ have cardinality $g$ and choose $i\in G$
  arbitrary.  By Lemma \ref{sec:supp-mark-basis-card-lemma} we find a
  cylinder set $\set{X_{G \setminus \set{i}} = y_{G\setminus\set{i}}}$
  which is contained in $\mathcal{Y}_{\leq}$. We have
  \begin{equation}
    m(x) = \sum_{B \in \Delta^{c}} z^{B}e_{B}(x) \leq 0 
    \qquad x \in \mathcal{Y}_{\leq}.
  \end{equation}
  Summing up these equations over the cylinder set $\set{X_{G
      \setminus \set{i}} = y_{G\setminus\set{i}}}$ yields
  \begin{equation}
    \label{eq:sum-is-smaller-than-zero}
    \sum_{x \in \set{X_{G \setminus \set{i}} = 
        y_{G\setminus\set{i}}}} \sum_{B \in \Delta^{c}} z^{B}e_{B}(x) \leq 0.
  \end{equation}
  Note that this summation is in fact the computation of the marginal
  $m_{G\setminus\set{i}}$ evaluated at the value
  $y_{G\setminus\set{i}}$. Since $m \in \ker \pi_{\Delta}$, and
  $G\setminus\set{i} \in\Delta$, equality must hold
  in (\ref{eq:sum-is-smaller-than-zero}). We find that every term in
  the sum was already zero:
  \begin{equation}
    \sum_{B \in \Delta^{c}} z^{B}e_{B}(x) = 0 
    \qquad x \in \set{X_{G \setminus \set{i}} = 
      y_{G\setminus\set{i}}}
  \end{equation}
  We will now inductively show that $m = 0$.  Contained in $\set{X_{G
      \setminus \set{i}} = y_{G\setminus\set{i}}}$ we have a smaller
  set $\set{X_{G} = y_{G}}$. Summing up the respective components of
  $m$ for this set we find, using Lemma
  \ref{sec:supp-mark-basis-eA-lemma}, 
  \begin{equation}
  \begin{aligned}
    0 & = \sum_{x \in \set{X_{G} = y_{G}}}\sum_{B \in \Delta^{c}}
    z^{B}e_{B}(x) \\
     & = z^{G} 2^{n-g} e_G(x_G)
  \end{aligned}
  \end{equation}
  It follows that $z^{G} = 0$. Applying the same argument, we can
  show that all coefficients $z^{H}$ vanish for $\abs{H} = g$. 
  Inductively, we continue with sets of cardinality $g+1$. Finally,
  this argument yields that all coefficients vanish and $m$ is
  zero. The whole procedure applies, mutatis mutandis, for the negative
  components as well. 
\end{proof}

Lemma \ref{sec:positive-support-is-bounded-from-below} completes the
proof of Theorem \ref{sec:main-theorem-thm} in the binary case. It
shows that the degree of each Markov move is at least $2^{g-1}$. Since
in fact we have a lower bound for the support, the degree bound can
only be realized by square free binomials.

\subsection{The non-binary case}
\label{sec:non-binary-case}

We now study the non-binary case. Let $\mathcal{X} = \bigtimes_{i\in N}
\mathcal{X}_{i}$ be some arbitrary, finite configuration space.

\begin{definition}
  Let $\phi_{i} : \mathcal{X}_{i} \to \set{0,1}, i\in N$ be surjective
  maps. For each $B \subset N$, the composed maps
  \begin{equation}
  \begin{aligned}
    \label{eq:collapsing_definiton}
    \phi_{B} : \mathcal{X}_{B} & \to \set{0,1}^{B}
    \\
    x_{B} & \mapsto (\phi_{i}(x_{i}))_{i\in B}.
  \end{aligned}
  \end{equation}
  are called \emph{collapsing maps}. Abbreviating, put $\phi \defas
  \phi_{N}$.  We have an induced map on contingency tables:
  \begin{equation}
    \label{eq:induced_collapsing}
    \begin{aligned}
      \Phi : \mathbb{N}_{0}^{\mathcal{X}} & \to
      \mathbb{N}_{0}^{\set{0,1}^{N}} \\
      (u(x))_{x\in \mathcal{X}} & \mapsto \left( \sum_{w\in
          \phi^{-1}(z)} u(w) \right)_{z \in \set{0,1}^{N}}.
    \end{aligned}
  \end{equation}
\end{definition}
The key property of such a collapsing is that it commutes with
marginalization.
\begin{lemma}
  \label{sec:phi_identity_lemma}
  Let $u \in \mathbb{N}_{0}^{\mathcal{X}}$. For $B\subset N, z_{B} \in
  \set{0,1}^{B}$ it holds:
  \begin{equation}
    \sum_{x_{B} \in \phi_{B}^{-1}(z_{B})} \sum_{w\in \set{X_{B} =
        x_{B}}} u(w) = \sum_{y \in \set{X_{B} = z_{B}}} \sum_{w \in
        \phi^{-1} (y)} u(w).
  \end{equation}
\end{lemma}
Note that for the cylinder set on the left hand side,
$\set{X_{B}=x_{B}} \subset \mathcal{X}$, while on the right hand side
$\set{X_{B}=z_{B}} \subset \set{0,1}^{N}$.
\begin{proof}
  Since on each side, every $w$ appears at most once, it suffices to
  show the set equality
  \begin{equation}
    \label{eq:set_equality}
    \bigcup_{x_{B} \in \phi_{B}^{-1}(z_{B})} \set{X_{B} = x_{B}} =
    \bigcup_{y \in \set{X_{B} = z_{B}}} \set{\phi^{-1} (y)}.
  \end{equation}
  \begin{description}
  \item[``$\subset$'':] Let $w$ from the left hand side be given. One
    has $X_{B}(w) = x_{B}$ for some $x_{B}$ with $\phi_{B}(x_{B}) =
    z_{B}$. Therefore $\phi(w) = y$ with $X_{B}(y) = z_{B}$ and $w$ is
    contained in the right hand side.  
  \item[``$\supset$'':] Let $w = \phi^{-1}(y), y\in \set{X_{B} =
      z_{B}}$ from the right hand side be given. We have $X_{B}(w) \in
    \phi^{-1}(z_{B})$, so $w$ is contained in the left hand side. 
  \end{description} 
\end{proof}

\begin{lemma}
\label{sec:marginalizationlemma}
Let $u,v \in \mathbb{N}_{0}^{\mathcal{X}}$ be contingency
tables. Denote the marginal map in the non-binary model as
$\pi_{\Delta}$, the corresponding binary one as $\rho_{\Delta}$. In
this case, $\pi_{\Delta}(u) = \pi_{\Delta}(v)$ implies $\rho_{\Delta}(\Phi(u)) =
\rho_{\Delta}(\Phi(v))$.
\end{lemma}
\begin{proof}
  Let $B \in \Delta, z_{B} \in \set{0,1}^{N}$. We have to show that 
  \begin{equation}
    \label{eq:1}
    \sum_{y \in \set{X_{B} = z_{B}}} \Phi(u)(y) = 
    \sum_{y \in \set{X_{B} = z_{B}}} \Phi(v)(y).
  \end{equation}
  By definition this equation is 
  \begin{equation}
    \label{eq:3}
    \sum_{y \in \set{X_{B} = z_{B}}}
      \sum_{w \in \phi^{-1} (y)} u(w)
    = \sum_{y \in \set{X_{B} = z_{B}}}
      \sum_{w \in \phi^{-1} (y)} v(w).
  \end{equation}
  Using Lemma \ref{sec:phi_identity_lemma} and the hypothesis, the
  statement follows.
\end{proof}
\begin{proof}[Theorem \ref{sec:main-theorem-thm}]
  Using the collapsing map, from generators of the non-binary model we
  can construct relations in the corresponding binary model as
  follows. Consider the polynomial rings $\mathfrak{R} \defas
  \mathbb{C}[p_{x}: x\in \mathcal{X}]$ and $\mathfrak{Q} \defas
  \mathbb{C}[q_{z} : z \in \set{0,1}^{N}]$. Given a simplicial complex
  $\Delta$, denote $\mathcal{I}_{\Delta} \subset \mathfrak{R}$ the
  non-binary toric ideal and $\mathcal{J}_{\Delta} \subset
  \mathfrak{Q}$ the binary one.  

  To each binomial $p^{m^{+}} - p^{m^{-}} \in \mathfrak{R}$ associate
  the collapsed binomial $q^{\Phi(m^{+})} - q^{\Phi(m^{-})} \in
  \mathfrak{Q}$. By Lemma \ref{sec:marginalizationlemma} it is clear
  that elements in the toric ideal $\mathcal{I}_{\Delta}$ are mapped
  to $\mathcal{J}_{\Delta}$. Furthermore, the support of
  $q^{\Phi(m^{+})}$, and $q^{\Phi(m^{-})}$ respectively, will have
  smaller cardinality than the supports of $p^{m^{+}}$ and $
  p^{m^{-}}$. Finally, if the non-binary model had a generator
  violating the statement of the theorem, then we can choose the maps
  $\phi_{i} : i\in N$ in such a way that this generator gets mapped to
  a non-zero binomial which violates the statement for the binary
  case. This contradiction concludes the proof. 
\end{proof}

\section{Neighborliness}
\label{sec:neighb-an-appl}
Before stating the neighborliness property we will take another short
excursion to statistics, introducing so called exponential families
and their relation to marginal polytopes.

Let again $\Delta$ denote a simplicial complex. For each
$x\in\mathcal{X}$ we have $A_{x}$ the corresponding row of the
marginal matrix $A_{\Delta}$, as defined in
\eqref{eq:A-matrix-defn}. The exponential family associated to this
complex is the parametrized family of probability measures
\begin{equation}
\label{eq:exponential family}
\mathbb{R}^{\mathcal{X}} \supset \mathcal{E}_{\Delta} 
\defas \set{ p_{\theta}(x) = Z(\theta)^{-1} \exp
  \left(
    \left\langle \theta,A_{x} \right\rangle \right) : \theta \in
  \mathbb{R}^{d}
}.
\end{equation}
Here, $Z(\theta) \defas \sum_{x\in\mathcal{X}} \exp \left(
  \left\langle \theta,A_{x} \right\rangle \right)$ is a normalization,
called the partition function. Like in Section \ref{sec:introduction},
$d$ is the number of rows of $A_{\Delta}$. By construction an
exponential family is an open subset of the simplex of all probability
measures on $\mathcal{X}$. Typically one is interested in the closure
$\overline{\mathcal{E}_{\Delta}}$, which is taken with respect to the
usual topology of $\mathbb{R}^{n}$. The closure of
$\mathcal{E}_{\Delta}$ equals the non negative part of the toric
variety $V(\mathcal{I}_{\Delta})$ \cite[Theorem
3.2]{geigermeeksturmfels06}. By this fact, the Markov basis gives the
implicit equations, cutting out the set
$\overline{\mathcal{E}_{\Delta}}$.

We can now state our main result in two equivalent formulations:
\begin{theorem}
\label{sec:neighbor_theorem}
  Let $g$ be the minimal cardinality among the non-faces of $\Delta$.
  \begin{description}
  \item[Geometric Formulation:] The marginal polytope is
    $2^{g-1}-1$ neighborly.
  \item[Probabilistic Formulation:] Every probability measure $p$ with
    $\abs{\supp (p)} < 2^{g-1}$ is contained in
    $\overline{\mathcal{E}_{\Delta}}$.
  \end{description}
\end{theorem}

\begin{proof} The probabilistic formulation is easy to see. Just observe
  that by Theorem \ref{sec:main-theorem} each monomial appearing in
  the set of generators $\set{p^{m^+} - p^{m^-} : m \in M} $ has
  cardinality of its support bounded from below by
  $2^{g-1}$. Therefore a $p$ with $\abs{\supp (p)} < 2^{g-1}$ must
  fulfill the defining equations trivially.

  Now, the geometric formulation is due to the well known fact that a
  set $\mathcal{Y} \subset \mathcal{X}$ is the support set of some $p
  \in \overline{\mathcal{E}_{\Delta}}$ if and only if $\conv \set{
    A_{y} : y \in \mathcal{Y}}$ is a face of the marginal polytope
  $Q_{\Delta}$. This is a consequence of the fact that the marginals
  computed by $A_{\Delta}$ form a sufficient statistics for the
  exponential family $\mathcal{E}_{\Delta}$. 
\end{proof}  

\begin{remark}[The bound is sharp]
  On first sight one would maybe expect a better neighborliness
  property in the non-binary cases, for instance if every variable is
  ternary. However, one can easily see that the bound is sharp in the
  sense that already for the ``no-three-way-interaction'' model with
  ternary variables, given by $N=\set{1,2,3},
  \mathcal{X}_{i}=\set{0,1,2}$ for $i=1,2,3$ and $\Delta = \set{ B
    \subset \set{1,2,3} : \abs{B} \leq 2}$, one has square-free
  generators of degree 4. They can easily be computed with 4ti2
  \cite{4ti2} or looked up in the Markov Bases Database
  \cite{MBDB}. Then a $p$ supported exactly on the positive support
  is a counterexample for any improvement of Theorem
  \ref{sec:neighbor_theorem}.
\end{remark}

\begin{remark}[Maximizing Multiinformation]
  The so called Multiinformation is an entropic quantity which
  generalizes mutual information to more than two variables. Denoting
  $H(p) \defas -\sum_{x\in \mathcal{X}} p(x)\log p(x)$ the entropy of
  $p$, and $H_{i}(p) \defas -\sum_{x\in \mathcal{X}_{i}} p_{\set{i}}(x)\log
  p_{\set{i}}(x)$ the marginal entropy for $i\in N$, it is defined as
\begin{equation}
  \label{eq:multiinfodef}
  MI(p) \defas \sum_{i\in N} H_{i}(p) - H(p)
\end{equation}
An interesting problem, considered in \cite{ayknauf06}, is to maximize
this function. There, all global maximizer in the binary case are
classified giving there support sets. In particular, by \cite[Theorem
3.2]{ayknauf06} all global maximizer $p^{*}$ satisfy
\begin{equation}
  \label{eq:supportbound}
  \abs{\supp (p^{*})} = 2. 
\end{equation} Let $\Delta_{2} \defas \set{B\subset N : \abs{B} \leq 2}$
denote the uniform simplicial complex of order two, then it is shown
\begin{corollary}[Theorem 3.5 in \cite{ayknauf06}]
  All global maximizer of $MI$ are contained in
  $\overline{\mathcal{E}_{\Delta_{2}}}$.
\end{corollary}
In view of the bound on the cardinality of the support, this now also
follows from our Theorem \ref{sec:neighbor_theorem}.
\end{remark}

\section{Markov Bases of high dimensional models}
\label{sec:markov-bases-high}
Finally, in this last section, we will show an example where the moves
$m_{G}^{y_{N\setminus G}}$ already constitute the full Markov
Basis. Consider again the binary case $\mathcal{X} =
\set{0,1}^{n}$. Let $G \subset N$. We denote
\begin{equation}
\label{eq:minus-interval-complex}
  \Delta_{\slash G} \defas \set{B \subset N : B \not\supset G},
\end{equation}
the complex of all sets not containing $G$. We have seen that the
toric ideal for this complex is generated in degree at least
$2^{\abs{G}-1}$.  In this section we show, that if $\Delta$ has the
structure (\ref{eq:minus-interval-complex}) and the variables are
binary, the Markov basis is given by the moves $m^{y_{N\setminus
    G}}_{\mathcal{G}}$ as defined in (\ref{eq:markov-moves}), and
therefore $I_{\Delta_{\slash G}}$ is generated in exactly degree
$2^{\abs{G}-1}$. As the no-three-way interaction model is of the form
(\ref{eq:minus-interval-complex}) it is also clear the statement of
the following theorem does not hold as soon as the variables are not
binary.

\begin{theorem}
  Let $\mathcal{G} = \left[G,N\right]$. 
  A Markov basis of the binary hierarchical model given by
  $\Delta_{\slash G}$ is 
  \begin{equation}
    M \defas \set{m_{\mathcal{G}}^{y_{N\setminus G}} : y_{N \setminus G}
      \in \mathcal{X}_{N\setminus G }}.
  \end{equation}
\end{theorem}
\begin{proof}
  We apply the standard technique\cite{develinsullivant03} of reducing
  the degree of a given binomial via the moves in $M$. For convenience
  we introduce tableau notation\cite{hostensullivant02a}
  for monomials. In this notation, the monomial $p^{u}$ is
  represented by listing each $x \in \mathcal{X}$,  $u(x)$ times. For
  example $p_{000}p_{110}p_{111}^{2}$ will be written as the tableau
  \begin{equation}
    \begin{bmatrix}
      0 0 0 \\
      1 1 0 \\
      1 1 1 \\
      1 1 1 \\
    \end{bmatrix}.
  \end{equation}

  Assume $p^{u}-p^{v} \in \mathcal{I}_{\Delta_{\slash G}}$.  Without
  loss of generality we assume that $G = \set{l,\ldots,n}$.  We can
  assume that $u$ and $v$ have disjoint supports, otherwise we write
  $p^{u}-p^{v} = q (p^{u'} - p^{v'})$ and the following argument shows
  that $p^{u'} - p^{v'}$ can be expressed in terms of the Markov
  basis. 
  Consider first the case $u(00\ldots0) \geq 1$. Since
  the marginals on the $n-1$ sets
  \begin{equation}
    \label{eq:marginalsets}
    \set{1,2,\ldots,n-1},\set{1,2,\ldots,n-2,n},
    \ldots,\set{1,2,\ldots,l-1,l+1,\ldots,n}
  \end{equation}
  of $u$ and $v$ coincide, and $v(00\ldots0) =0$ we find that the
  given binomial has the form
  \begin{equation}
    \begin{bmatrix}
      00\ldots 0\underline{000\ldots0} \\
      \ldots \\
      \vdots \\
      \ldots \\
      \vdots
    \end{bmatrix}
    - 
    \begin{bmatrix}
      00\ldots 0\underline{100\ldots 0} \\
      00\ldots 0\underline{010\ldots 0} \\
      \vdots \\
      00\ldots 0\underline{000\ldots 1} \\
      \vdots
    \end{bmatrix}
  \end{equation}
  where the set $G$ is underlined.  Applying the same argument in the
  other direction, namely that, since $u$ has the same $n-1$ marginals
  on the sets (\ref{eq:marginalsets}) we find that $u(x)>0$ for any
  $x$ which has exactly two non-zero positions, both lying in $G$,
  formally $u(x)>0$ for any $x$ with $\supp(x) \subset G$ and
  $\abs{\supp(x)}=2$. We continue to find that $v(x) > 0$ for any $x$
  with $\supp(x) \subset G$ and $\abs{\supp(x)} = 3$. Repeating this
  argument we find that $p^{u}$ contains all configurations with zero
  outside $G$ and an even number of ones in $G$. Conversely, $p^{v}$
  contains all configurations that are zero outside $G$ and have
  an odd number of ones in $G$. All together, this is exactly the move
  $m_{\mathcal{G}}^{00\ldots0}$.  Obviously, in the general case, if
  in the beginning we would have started with some other configuration
  instead of $00\ldots0$, say $y$, the same argument leads to the move
  $m_{\mathcal{G}}^{y_{N\setminus G}}$ instead. Abbreviating the
  specific move as $m$ now, we write $p^{u} = K p^{m^{+}}$ and $p^{v}
  = L p^{m^{-}}$ with some monomials $K,L$ and have
  \begin{equation}
    \begin{aligned}
      p^{u} - p^{v} & = K p^{m^{+}} - L p^{m^{-}} + K p^{m^{-}} - K p^{m^{-}} \\
      & = K (p^{m^{+}} - p^{m^{-}}) - (L - K) p^{m^{-}}.
    \end{aligned}
  \end{equation}
  The degree of $L-K$ is obviously smaller than the degree of
  $p^{u}-p^{v}$. Inductively it follows that $p^{u}-p^{v}$ can be
  written as combination of the moves $m_{\mathcal{G}}^{y_{N\setminus
      G}}$.
\end{proof}

\section*{Acknowledgment}
\label{sec:acknowledgment}
The author is supported by the Volkswagen Foundation. Furthermore, he
wishes to express his gratitude to Johannes Rauh and Nihat Ay for
many valuable discussions.

\providecommand{\bysame}{\leavevmode\hbox to3em{\hrulefill}\thinspace}
\providecommand{\MR}{\relax\ifhmode\unskip\space\fi MR }
\providecommand{\MRhref}[2]{%
  \href{http://www.ams.org/mathscinet-getitem?mr=#1}{#2}
}
\providecommand{\href}[2]{#2}

\end{document}